\documentclass{amsart}

\usepackage{amssymb}
\usepackage{enumerate}
\usepackage{ifpdf}

\ifpdf
\usepackage[protrusion=true,expansion,kerning,letterspace=-1000]{microtype}
\fi

\newtheorem{theorem}{Theorem}[section]
\newtheorem{lemma}[theorem]{Lemma}
\newtheorem{corollary}[theorem]{Corollary}
\newtheorem{proposition}[theorem]{Proposition}

\theoremstyle{definition}
\newtheorem{definition}[theorem]{Definition}

\theoremstyle{remark}
\newtheorem{remark}[theorem]{Remark}

\newcommand{\restrictedTo}[1]{\raisebox{-1pt}{\big{|}}_{#1}}
\newcommand{\cl}{\textup{cl}}
\newcommand{\rec}{r.e.c.\,}

\begin{document}

\date{\today}
\title[On Representation of the Reeb Graph]{On Representation of the Reeb Graph 
\\ as a Sub-Complex of Manifold}

\author[M. Kaluba]{Marek Kaluba}
\address{\textup{Marek Kaluba:}
Adam Mickiewicz University in Pozna{\'n}\\
Faculty of Mathematics and Computer Science\\
ul. Umultowska 87\\
61-614 Pozna{\'n}, Poland
} 
\email{kalmar@amu.edu.pl}

\author[W. Marzantowicz]{Wac{\l}aw Marzantowicz$^*$}
\address{\textup{Wac{\l}aw Marzantowicz:}
Adam Mickiewicz University in Pozna{\'n}\\
Faculty of Mathematics and Computer Science\\
ul. Umultowska 87\\
61-614 Pozna{\'n}, Poland
} 
\email{marzan@amu.edu.pl}

\thanks{$^*$Supported  by the Polish Research Grant NCN 2011/03/B/ST1/04533}

\author[N. Silva]{Nelson Silva$^\dag$}
\address{\textup{Nelson Silva:}
Adam Mickiewicz University in Pozna{\'n}\\
Faculty of Mathematics and Computer Science\\
ul. Umultowska 87\\
61-614 Pozna{\'n}, Poland
} 

\curraddr{
Institute of Mathematics and Computer Science\\
University of São Paulo\\
Avenida Trabalhador S\~{a}o-carlense, 400 - Centro\\
CEP: 13566-590, S\~{a}o Carlos, SP Brazil}
\email{nelson@icmc.usp.br}

\thanks{$^\dag$Supported  by FAPESP of Brazil Grant BEPE 2012/15659-5}

\begin{abstract}
The Reeb graph  $\mathcal{R}(f) $ is one of the fundamental invariants of a 
smooth function $f\colon  M\to \mathbb{R} $ with isolated critical points. It 
is 
defined as the quotient space  $M/_{\!\sim}$ of the closed manifold $M$ by a 
relation that depends on $f$. Here we construct a $1$\nobreakdash-dimensional 
complex $\Gamma(f)$ embedded into $M$ which is homotopy equivalent to 
$\mathcal{R}(f) $. 
As a consequence we show that for every function $f$ on a manifold with finite 
fundamental group, the Reeb graph of $f$ is a tree. If $\pi_1(M)$ is 
an abelian group, or more general, a discrete amenable group, then 
$\mathcal{R}(f)$ 
contains at most one loop. Finally we prove that the number of loops in the 
Reeb graph of every function on a surface $M_g$ is estimated from above by 
$g$, the genus of $M_g$.
\end{abstract}

\maketitle

\section{Introduction}

The Reeb graph $\mathcal{R}(f) $ of a function $f\colon  M\to \mathbb{R}$  had
been defined more than 60 years ago in \cite{Reeb} and
\cite{Kronrod}, but just recently it has attracted more attention.
It plays a fundamental role in computational topology for shape
analysis (\cite{Biasotti,Cole-Mc}). Examples of applications of
Reeb graphs include manifold reconstruction, 3D object indexing, 3D
object recognition and many more. It is a very accessible invariant of the 
pair $(M,f)$  giving a simplification of the topological space $M$ and hence 
much desired from the computational point of view (cf. \cite{Dey}). On the other
hand, every graph which has no (oriented) cycles is represented as the Reeb 
graph of a $C^1$\nobreakdash-function $f\colon  M\to \mathbb{R}$ on a surface 
$M$, see 
\cite{Masumoto,Sharko}. This gives a converse relation from graphs to pairs 
$(M,f)$ and is done by two-dimensional surgery argument.

In this paper our aim is to provide a construction of the Reeb graph with a 
simplicial structure and as a subspace of $M$. We achieve it by constructing 
a finite one-dimensional complex $\Gamma(f) \subset M$ which is homotopy 
equivalent to $\mathcal{R}(f) $. The sub-complex is constructed via spaces of 
paths connecting critical points of $f$. The appropriate homotopy relation 
creates a bijective correspondence between the equivalence classes of paths and 
simplices of $\mathcal{R}(f) $.

Let us denote by $\pi\colon (M,f)\to \mathcal{R}(f) $ the canonical projection 
from 
the manifold to the Reeb graph, and let $\iota\colon\Gamma(f) \hookrightarrow 
M$ denote the inclusion. We prove that the composition  \[\pi\circ\iota\colon 
\Gamma(f)\to \Gamma(f)\cong \mathcal{R}(f) \] induces the identity on the 
fundamental group of $\mathcal{R}(f) $ and this allows us to obtain a few 
relations between fundamental groups of $M$ and $\mathcal{R}(f) $. In 
particular 
we show that if $M$ is
simply-connected, or more general $\pi_1(M)$ is finite, then
$\mathcal{R}(f) $ is a tree, i.e. a contractible finite graph.
By the same argument we prove that if $\pi_1(M)$ is abelian, or more
general if $\pi_1(M)$ does not contain  $\mathbb{F}_2$, the free group on two
generators, then $\mathcal{R}(f) $ contains at most one loop. This class of 
groups contains discrete amenable groups, thus nilpotent groups, solvable 
groups and many others (see \cite{Nowak}). 

Finally, by a direct geometrical argument, we show that for a
$C^1$\nobreakdash-fun\-ction $f\colon  M\to \mathbb{R}$ on a compact closed 
surface 
$M_g$ the number of loops of $\mathcal{R}(f) $ is less than or equal to $g$, the 
genus 
of $M$. This is a complementing result to  that of \cite{Cole-Mc} which says 
that for a Morse function $f\colon M\to \mathbb{R}$ the number of loops of 
$\mathcal{R}(f) $ is 
equal to $g$.

The paper is organised as follows.
In Section \ref{sec:Reeb-graph-bsic-props} we set notation and global 
assumptions that hold throughout the paper. We also provide classic definitions 
and results obtained before, and prove some generalisations for 
$C^1$\nobreakdash-functions. Section \ref{sec:conncted-compnts-&-pths} is 
devoted to establish basic correspondences between sets of connected 
components, 
sets of homotopy classes of paths and edges and vertices of the Reeb graph.
We join these results in Section \ref{sec:Graphs}, where we define and study 
$\Gamma(f)$. Finally, in Section \ref{sec:applications} we apply these results 
to some questions concerning Reeb graphs on classical maifolds and we estimate 
number of loops in $\mathcal{R}(f)$ for a surface $M_g$.

\section{Reeb graph and basic properties}\label{sec:Reeb-graph-bsic-props}

Throughout the paper $M$ is a smooth (i.e. $C^1$) closed connected manifold of 
dimension $n\geqslant 2$ and $f\colon M \to \mathbb{R}$ is a 
$C^1$\nobreakdash-function. 
Moreover we will assume that all critical points of $f$ are \textit{isolated}.

Since $M$ is a compact space, the set of all critical points of $f$
(denoted by $\textup{Cr}(f)\subset M$) and the set of all
critical values of $f$ (denoted by ${\textup{V}}_{cr}(f)\subset \mathbb{R}$) is 
finite.
For a given $c,c'\in \mathbb{R}$, we use the following notations:
\begin{align*}
M_c &= f^{-1}(c),& M^{(c,c')} &= \{x\in M\colon c< f(x)< c' \},\\
M^c &=\{x \in M\colon f(x)\leqslant c\},& 
M^{[c,c']} &=\{x\in M\colon
c\leqslant f(x)\leqslant c'\}. 
\end{align*}
As closed subsets, the sets $M_c,M^c,M^{[c,c']}$ are compact subspaces in 
$M$. The first set is called \textbf{the level set} of $f$. Note that
\[M =\,\bigcup_{{-\infty <c<\infty}} \, M_c.\]
For a given $c\in \mathbb{R} $ let $M^s_c$ denote a connected component of the 
level 
set $M_c$. Now we recall the classical definition of the Reeb graph (see 
\cite{Kronrod, Reeb}).

\begin{definition}\label{Reeb definition}
We say that points   $x,y\in M$ are in the Reeb relation $x \sim y$,
if $x$ and $y$ are in the same connected component $M_c^s$ of $M_c$.
The quotient space $ M/_{\!\sim}$ is called \textbf{the Reeb graph} of
$f$ and will be denoted by $\mathcal{R}(f) $.
\end{definition}
The Reeb graph is a well defined entity by the virtue of the following lemma.
\begin{lemma}[see \cite{Kronrod, Reeb, Sharko}]
\label{Reeb lemma} The quotient space $\mathcal{R}(f) $ is homeomorphic to the 
body of a finite graph. Moreover, the vertices of $\mathcal{R}(f) $ correspond
to the classes of $x\in \textup{Cr}(f)$, i.e. to connected components $M_c^s$ 
such that $M_c^s\cap \textup{Cr}(f)\neq \varnothing$. Furthermore, $f$ induces 
a function
$\widetilde{f}\colon\mathcal{R}(f)  \to \mathbb{R}$ such that 
$f=\widetilde{f}\circ\pi$, which is strictly monotonic on each edge of 
$\mathcal{R}(f) $.
\end{lemma}

\begin{definition}\label{essential components}
Let $c\in V_{cr}(f)\subset \mathbb{R}$ be a critical value and let $M_c^s$ be a 
connected component of $M_c$. We call $M_c^s$ an \textbf{essential component} 
of $M_c$ if $M_c^s\cap \textup{Cr}(f) \neq \varnothing$. In other words,
$M_c^s$ is essential if it contains a critical point. In the case
$M_c^s$ is denoted by $M_c^{es}$.
\end{definition}

\begin{proposition}\label{prop:connected-componets}
Suppose that $M$ is a smooth compact manifold and that $f\colon M\to 
\mathbb{R}$ is a 
$C^1$\nobreakdash-function with isolated critical points. Let $c$ be a 
critical value and let \[A=\{x_1, x_2,\dots, 
x_n\}=M^{es}_c \cap \textup{Cr}(f)\] be the set of critical points in 
$M^{es}_c$.
\begin{enumerate}
\item Every path-connected component of a level set is its connected 
component.
\item Every two different critical points $x_1,x_2$ in the closure 
of a connected component of $M^{es}_{c} \setminus A$ can be
connected by a path $\gamma\colon I \to M_c^{es}$ such that $\gamma(0)=
x_1$, $\gamma(1)=x_2$, and $\gamma$ is a homeomorphic embedding \textup{(}an 
arc\textup{)}.
\item  
There exists a closed subspace $K\subset M^{es}_c$ homeomorphic to a tree such 
that the set of vertices is equal to $A$.
\end{enumerate}
\end{proposition}

\begin{proof}
Since a path-connected 
component is always a connected component, one inclusion of the first statement 
is obvious. For the other inclusion note that it is enough to study essential 
components of a level set. Otherwise, the level set is a manifold, hence 
path-connected components are its connected components. 

Let $x_0\in M_c$ be a critical point and let $M_c^s = M_c^{es}$ denote its
connected component. It is enough to show that $M_c^{es}$ is locally
path-connected at each of its critical points. 

Note that if $M_c^{es}=\{x_0\}$ there is nothing to prove, so we may assume 
that $M_c^{es}\neq\{x_0\}$ and $M_c^{es}\setminus\{x_0\}$ is a
manifold locally near $x_0$. Indeed, for any $y\in M_c^{es}\setminus\{x_0\}$ 
let $U_y$ be an open set in $M$ such that
\begin{itemize}
 \item $U_y$ is diffeomorphic to an open disc in $\mathbb{R}^n$,
 \item $U_y$ does not contain any critical point of $f$,
 \item $U_y$ does not contain points of other connected component of $M_c$ 
besides $M_c^{es}$.
\end{itemize}
Since $f|_{U_y}\colon U_y\to \mathbb{R}$ is a submersion, ${U_y} \cap 
f^{-1}(c)\subset (M_c^{es}\setminus\{x_0\})\cap U_y$ is a submanifold.

Take a point $y \in  M_c^{es}\setminus\{x_0\}$ near $x_0$ and let
$\widetilde{M}_c^{es}(x_0,y)$ be a connected component of 
$M_c^{es}\setminus\{x_0\}$ containing $y$. We claim that \[x_0\in 
\cl\big(\widetilde{M}_c^{es}(x_0,y)\big),\] where $\cl(A)$ 
denotes the closure of $A$. 
Observe that $\widetilde{M}_c^{es}(x_0,y)$ is open and closed in
$M_c^{es}\setminus\{x_0\}$. Suppose that $x_0\not\in 
\cl\big(\widetilde{M}_c^{es}(x_0,y)\big)$. 
Then 
$\cl\big(\widetilde{M}_c^{es}(x_0,y)\big)=\widetilde{M}_c^{es}(x_0,y)$ 
in $M_c^{es}$, and $\widetilde{M}_c^{es}(x_0,y)$ is open in $M_c^{es}$, 
because $M_c^{es}\setminus\{x_0\}$ is open. Therefore, 
$\widetilde{M}_c^{es}(x_0,y)$ is a connected component of $M^{es} _c$ which 
does not contain $x_0$. Since $M_c^{es}$ is connected, 
$\widetilde{M}_c^{es}(x_0,y)$ has to be empty, which is a contradiction.

Consider a descending family of open disks $\{D_n\}_{n\in\mathbb{N}} \subset M$
centred at $x_0$. We can assume that $y\in \partial\,\cl(D_1)$, the boundary 
of $\cl(C_1)$. Observe that $\partial\,\cl(D_n)\cap 
\widetilde{M}_c^{es}(x_0,y)\neq 
\varnothing$, for every $n$. Otherwise we would cover a connected set 
$\widetilde{M}_c^{es}(x_0, y)$ by two open disjoint sets
\[A=\widetilde{M}_c^{es}(x_0,y)\cap D_n\quad 
\text{and}\quad B=\widetilde{M}_c^{es}(x_0,y)\setminus D_n.\]

Since $\widetilde{M}_c^{es}(x_0,y)$ is locally a manifold it is
path-connected. Set $x_1=y $. Let $x_n\in D_n \setminus D_{n+1} $ and
$x_{n+1}\in D_{n+1}\setminus D_{n+2} $ be any points. Then for every $n$ there 
is a path $\gamma_n\colon I\to \widetilde{M}_c^{es}(x_0,y)$ connecting
$x_n$ and $x_{n+1}$.

We define a path $\gamma\colon I\to \widetilde{M}_c^{es}(x_0,y)\cup 
\{x_0 \} = M_c^0$ by the formula
\[\gamma(t)= \begin{cases}
\gamma_1(t) & \quad \text{ if }\, 0\leqslant t \leqslant 
\!\!\,^1\!/_{\!2}, \\
\gamma_2(t) & \quad \text{ if }\,^1\!/_{\!2}\leqslant t \leqslant
\!\!\,^1\!/_{\!4}, 
\\ 
\vdots & \quad \quad\vdots \\
\gamma_n(t) & \quad \text{ if }\, 1-\!\,^1\!/_{\!2^n}\leqslant t\leqslant  
1-\!\,^1\!/_{\!2^{n+1}},\\
\vdots & \quad \quad\vdots \\
x_0         & \quad \text{ if } \; t=1.
\end{cases}
\]
Consequently every point $y\in M_c^{es}$ sufficiently close to $x_0$ can be 
connected to $x_0$ by a path. This shows that $M_c^{es}$ is locally 
path-connected, thus path-connected. Note also that all the paths $\gamma_n$ 
used in the construction of $\gamma$ can be taken as homeomorphic embeddings of 
an interval (i.e. arcs). Consequently path $\gamma\colon  I \to M^{es}_c $ 
$\gamma(0)=y$, $\gamma(1)=x_0$ is a homeomorphic embedding of $I$, i.e. an arc, 
as the infinite composition of such arcs.

To prove the second statement we take two different critical points 
$x_1,x_2 \in M^{es}_c$. Let $y_i$ for $i=1,2$ be two points in a 
connected component of $M^{es}_c \setminus \text {Cr}(f)$ such that $y_i$ is 
in a neighbourhood $U_i$ of $x_i$. Moreover assume that $U_{1}\cap 
U_2=\varnothing$. Choosing a sufficiently small neighbourhood $U_i$ we may 
connect $x_i$ and $y_i$ by an arc $\gamma_i$ contained in $U_i$ for $i=1,2$. 
Since  $y_1$ and $y_2$ are in the same component of $M^{es}_c 
\setminus\textup{Cr}(f)$ we can find an arc $\widetilde{\gamma} \subset  
M^{es}_c \setminus \textup{Cr}(f)$ connecting them. The composition of
$\gamma_1$, $\widetilde{\gamma}$, and $\gamma_2^{-1}$ is an arc connecting 
$x_1$ and $x_2$. To ease the proof of the third statement denote this arc 
$\gamma$ by $\gamma^1_2$.

To prove the third statement we use a procedure similar to constructing a 
spanning tree. However, a special care is needed since we desire a geometric 
embedding of the tree in $M^{es} _c$. In particular we require that edges are
realised by arcs in $M^{es}_c$ which do not intersect, except only in vertices.

Recall that $A=\{x_1, \ldots, x_n\}$ is the set of all critical points 
in $M^{es}_c$. Fix $x_1$ as the starting point. Let $A^1=\{x^1_1 ,x^1_2 
,\dots\}\subset A$ be the subset of $A$ containing points that can 
be connected with $x_1$ by arcs $\gamma^1_j$ as above such that $\gamma^1_j \cap
\gamma^1_{j^\prime} = \{x_1\}$ for $j\neq j^\prime$. Set
\[K^1 = \;\bigcup_{{x_j \in A^1}} \, \gamma^1_j.\]
Then $K^1$ is homeomorphic to a one-dimensional complex contractible to
$x_1$ (a hairy ball).

Apply consecutively the above construction to all $x^1_j$ in $A^1$ as the 
starting points. In each step connect $x^1_j$ analogously only to points in 
$A^1_j$, the set of these points in $A$ which have not appeared in the previous 
steps. We require that the joining arcs intersect only at $\{x^1_j\}$. 
Denote by $K^2_j$ such a one-dimensional complex contractible to $x_j$. Let
\[K^2 = K^1 \cup \; \bigcup_{{x_{j} \in A^1}}\, K^1_j,\]
and observe that $K^2$ is homeomorphic to a tree rooted at $x_1$.
Indeed, suppose that $K^{1} \cap K^1_j \neq\varnothing$, that is there is a 
cycle in $K^2$. This means that there are arcs $\gamma^1_i$ and $\gamma^j_k$ 
intersecting at a middle point $y$. Then \[\delta(t)=
\begin{cases}
\gamma^1_i(t) & \text{until $\gamma^1_i(t)$ reaches }y\\
\gamma^j_k(t) & \text{afterwards,}
\end{cases}
\]
is an arc connecting $x_1$ to $x^j_k$. Therefore $x^j_k$ would belong to both 
$A^1$ and $A^1_j$ which is a contradiction. 
Repeating this construction we get a compact set $K\subset M^{es}_c$ 
homeomorphic to a graph such that
\begin{itemize}
\item $\{x_1, \dots, x_n\} \subset K$,
\item all $x_i$ are vertices,
\item $K$ is contractible to $x_1$.
\end{itemize}
which proves the proposition. 
\end{proof}

\begin{corollary}
Any two different critical points $x_i,x_j\in A $ can be joined 
by an arc which is a composition of arcs as above.
\end{corollary}

\begin{remark}\label{finite connected}
Given a (regular or critical) value $c\in \mathbb{R}$ of $f$, the
topological space $M_c$ has a finite number of path-connected
components $M_c^s$. The family $\{M_c^s\}$ forms an open cover of
$M_c$ by disjoint sets, thus it is finite.
\end{remark}

\section{Connected components and classes of 
paths}\label{sec:conncted-compnts-&-pths}

In this section we establish a correspondence between homotopy classes of paths 
and connected components of $M^{(c,c')}$. We start with well known facts about 
gradient trajectories in $M$.
Then we prove that every connected component of type (IIa) (see Definition 
\ref{def:type}) corresponds bijectively to edges of $\mathcal{R}(f)$.
Later we consider three increasing in generality families of paths: 
decreasing edge-paths, edge-paths and extended edge-paths (see 
definitions \ref{def:edge-path} and \ref{def:extended edge-path}).
On these families we introduce an equivalence relation --
homotopy \rec (see Definition \ref{def:homotopy_rec}) and prove that under 
the relation the families are the same.

At the end we establish a bijection between classes of homotopy \rec of 
paths and connected components of $M^{(c,c')}$. The composition of 
bijections above allows us to identify edges (vertices) of $\mathcal{R}(f)$ 
and homotopy classes of extended edge-paths (contractible \rec paths, 
respectively).

These results will be needed in Section \ref{sec:Graphs}.

\subsection{Gradient trajectories}
\begin{definition}\label{def:trajectory-connects}
We say that a gradient flow-line $\gamma$ \textbf{connects} a critical point 
$x'\in 
M^{es} _{c'}$ to a point $x\in M^{es}_c$ if 
\begin{itemize}
 \item $\lim \gamma (t)=x'$ as $t\to -\infty$, 
\item if $x$ is a critical point then $x=\lim \gamma(t)$ as $t\to \infty$,
\item if $x$ is a regular point there exists $t_0$ such that $\gamma(t_0)=x$.
\end{itemize}

\end{definition}

\begin{remark}\label{continuous extension}
To every trajectory as in the definition above, there corresponds
a decreasing path denoted also by $\gamma\colon I\to M$, defined as 
follows. We take a monotonic smooth diffeomorphism from the real line (a 
closed ray, respectively) to the interval $(0,1)$ ($(0,1]$, 
respectively). The composition of the diffeomorphism with the trajectory map 
has a unique continuous extension to a path $\gamma\colon [0,1] \to \subset M$ 
such that $\gamma(0)=x'$, $\gamma(1)=x$ in each of the cases above.
\end{remark}

The next proposition allows to effectively construct edges of graph
$\Gamma(f)$ (to be defined in Section \ref{sec:Graphs}) as the integral curves 
of a differential equation. For an explanation of components of type (IIa) 
see Definition \ref{def:type}.

\begin{proposition}\label{prop:joining critical levels}
Let $f: M\to \mathbb{R}$ be a $C^1$\nobreakdash-function with finite number of
critical points. Suppose that $C$ is a component of type (IIa) of $M^{(c,c')}$ 
such that $M^{es} _{c'} =M_{c'} \cap \cl(C)$.
Then every point $x_{c'} \in M_{c'}^{es} \cap \cl(C)$ can be connected with a 
point in $M_{c} ^{es}$ by the negative gradient trajectory $\gamma \subset 
\cl(C)$.
\end{proposition}

\begin{proof} The proposition follows directly from the classical theorem 
stated below.
\end{proof}

 \begin{theorem}[cf.\,{\cite[Chapter 1.6]{Kat}}]\label{thm:gradient flow}
 Let $M$ be a compact closed manifold equipped with a Riemannian
 structure, and   $f$ be a $C^1$\nobreakdash-function. Let $-\nabla f
 (x)\colon M \to TM$ be the negative gradient vector field defined by $f$ and 
the Riemannian structure. Then the following conditions are satisfied.
 \begin{enumerate}
 \item The field $-\nabla f$ is orthogonal to the level sets.
 \item The $\omega$-limit \textup{(}positive\textup{)} set $\omega_{-\nabla 
f}(x) $, and the 
$\alpha$-limit \textup{(}negative\textup{)} set $\alpha_{-\nabla f}(x)$ consist 
of critical points of $f$, i.e. are fixed points of the gradient flow.
 \item For any $x\in M$ and any $f$ the set $\omega_{-\nabla f}(x)$ is either a 
single point or an infinite set. The same holds for $\alpha_{-\nabla f}(x)$.
 \item If the function $f$ has only isolated critical points then every 
1
non-trivial trajectory of  $-\nabla f(x)$ converges to a critical point of $f$ 
as $t\to \pm\infty$.
 \end{enumerate}

 \end{theorem}

\begin{remark}Suppose that $\gamma$ connects $x'\in M^{es} _{c'} $ to $x\in 
M^{es}_c$. Even if $x'$ is a critical point, $x$ may be a regular point.
In particular, not every pair of critical points (in different level sets) are 
connected by a negative gradient trajectory (see \cite[Chapter 2.3]{Nitecki} 
for 
an example of a height function of a slightly inclined horizontal torus). 
For decreasing paths connecting critical points see Lemma 
\ref{lem:hmtpy-to-pth-crit2crit}.
\end{remark}

\subsection{Connected components.}
\begin{definition}\label{def:type}
Recall that $M^{(c,c')}$ was defined as $\{x\in M 
\colon f(x)\in (c,c')\}$.
\begin{itemize}
 \item A connected component of $M^{(c,c')}$ is called of \textbf{type (I)} if 
it
contains a critical point. Otherwise, it is called of \textbf{type (II)}.
\item Let $C$ be a component of type (II) of $M^{(c,c')}$. Suppose that 
\[f(\cl(C))\cap \text{V}_{cr}(f)= \{c,c'\},\]
i.e. $M_c^{es}\cap\cl(C)\neq \varnothing$ and $M_{c'}^{es}\cap 
\cl(C)\neq \varnothing$. Then we call $C$ a component of {\textbf{type
(IIa)}} or an {\textbf{edge-component}}.
\end{itemize}
\end{definition}

In other words, component of type (IIa) is a connected component of 
$M^{(c,c')}$ which crosses every level set $M^{(c,c')}$ and its boundary 
intersects some essential components of $M_c$ and $M_{c'}$.

\begin{proposition}\label{prop:gradient_product_diffeo}
Let $C$ be a component of type (IIa) of $M^{(c,c')}$. 
There exists a diffeomorphism 
\[\left(M_a \cap C\right)\times (c,c')\cong M^{(c,c')}\cap C.\]
\end{proposition}

\begin{proof}
Note that $g=f\restrictedTo{C}$ is a function without critical points.
Consider the vector field $X = -\nabla g$. Since $\|X(p)\|>0$ for all
$p\in C$, we can define $Y = \frac{X}{\|X \|}$ on $C$. In
particular, it is defined on $g^{-1}([a,b])=M^{[a,b]}\cap C$ for any $a,b$ such 
that $c<a<b<c'$. Let $\phi_p(t)$ denote the integral curve of $Y$ passing 
through $p\in g^{-1}(a)$. Since \[\frac{d\,\phi_p(t)}{dt}=1,\]
the flow of $Y$ which starts from the level $a$ at time $0$, will reach
the level $b$ at time $b-a$. Therefore we obtain a diffeomorphism
\[g^{-1}(a)\times [a,b]\ni (p,t)\longmapsto h(p,t)=\phi_p(t-a)\in
g^{-1}([a,b]).\]

Now we want to extend the diffeomorphism to the whole connected component $C$.

Denote by  $D$ the connected component $M_d\cap C$ of the level set $M_d$ for 
some $d\in (c,c^\prime)$. Consider an ascending family of 
intervals $\{[a_n,b_n]\}$, such that
\[\bigcup_n [a_n,b_{n}] = (c,c') \quad \text{and} \quad d\in \bigcap_n 
[a_n,b_n].\]
Next we obtain a diffeomorphism
\[
h_n\colon D\times [a_n,b_n] \to g^{-1}([a_n,b_n]) 
\]
in a similar way as $h$ was obtained.
By the uniqueness of the integral curve passing through a point, we can define 
a diffeomorphism as the direct limit
\begin{align*}
\varinjlim_n h_n \colon D\times\bigcup_n[a_n,b_n]=D\times(c,c') & 
\longrightarrow 
\bigcup_n g^{-1}\big([a_n,b_n]\big)=C
\end{align*}
by the formula
\begin{align*}
(p,t) & \longmapsto h_n(p,t) \quad \text{for some $n$}.
\end{align*}
\end{proof}

\begin{lemma}\label{connectenes type II}
A component $C$ of type (II) of $M^{(c,c')}$ intersects exactly 
one connected component of $M_d$, for each level set $d\in(c,c')$.
\end{lemma}

\begin{proof}
Observe that $f\restrictedTo{C}\colon C\to \mathbb{R}$ does not have any 
critical 
point. Let $ 
c<d < d^\prime <c^\prime$. By the gradient flow argument (see Proposition 
\ref{prop:gradient_product_diffeo}) $M^{[d,d^\prime]}$  is homeomorphic to 
$I\times M_d $, 
 and simultaneously to $I \times M_{d ^\prime} $ which sets a bijection between 
connected components of $M^{[d,d']}$ and connected components of $M_d$ for all 
$d\in (c,c')$. Since $C$ intersects exactly one connected component of 
$M^{[d,d']}$, it intersects exactly one component of $M_d$.
\end{proof}

\begin{corollary}\label{cor:finite-numbr-contcd-compnts}
The manifold $M^{(c,c')}$ has a finite number of connected
components.
\end{corollary}

\begin{proof}
The manifold $M^{(c,c')}$ decomposes as a union of type (I) and type (II) 
components. The number of critical points is finite, hence we have a finite
number of components of type (I). Therefore it suffices to estimate the number 
of type (II) components.

Fix a regular value $d\in (c,c')$. Let $C$ be component of type (II). Note that 
$M_d$ has finite number of components as a compact sub-manifold of $M$. 
Since there are only finitely many of choices of $d$ which correspond to 
a different (non-canonically homeomorphic) level sets, the 
statement follows from the proof of Lemma \ref{connectenes type II}.
\end{proof}

Let us denote the interior of $J$ by $\int (J)$ .

\begin{proposition}\label{prop:compIIa_is_edge}
There is a bijection between components of type (IIa) in $M$ and edges of 
$\mathcal{R}(f) 
$.
\end{proposition}

\begin{proof}
The easy part is to show that $C$, a component of type (IIa) of $M^{(c,c')}$ 
satisfies \[\pi(C)=\int (J)\] for some $1$\nobreakdash-simplex 
$J\in\mathcal{R}(f)$. 
We just need to observe that \[[c,c']=f(\cl(C))=\widetilde{f}\circ 
\pi(\cl(C))=\widetilde{f}(J).\]
Therefore $\pi$ maps $\cl(C)$ bijectively to a set in $\mathcal{R}(f)$ 
satisfying: 
\begin{itemize}
 \item $\widetilde{f}\circ\pi\big(\cl(C)\cap M_c\big)= c$,
 \item $\widetilde{f}\circ\pi\big(\cl(C)\cap M_c'\big)= c'$,
 \item no point in $\pi(C)$ is mapped via $\widetilde{f}$ to a 
critical value.
\end{itemize}
Since $C$ is connected, this is an edge of the Reeb graph.

Conversly, let $J$ be an edge in $\mathcal{R}(f)$ and let $x\in 
\pi^{-1}(\int(J))$. Set $d=f(x)$ and consider $M_{d}(x)$, the 
connected component of the level set $M_{d}$ to which $x$ belongs. Since $f$ 
restricted to $\pi^{-1}(\int(J))$ has no critical points, Lemma 
\ref{connectenes type II} implies that there exists a 
component $C_i$ of type (II) which contains $M_d(x)$. Define $C= 
\bigcup_i C_i$ to be the maximal (under inclusion) component satisfying these 
conditions. We claim that $\pi(C)=\int(J)$ and that $C$ is of type (IIa). 

Observe that by definition $\pi(C_i)\subset \int(J)$. Suppose that there 
exists a point $y\in \int(J)$ such that $\pi^{-1}(y)\cap C= \varnothing$. Then 
using Proposition \ref{prop:gradient_product_diffeo} we can show that 
$\pi^{-1}(y)$ is diffeomorphic to $\pi^{-1}(d)$. As $\pi^{-1}( (y,d) )$ is 
diffeomorphic to a cylinder, $C$ can be extended further to cover 
$\pi^{-1}(y)$. This contradicts the definition of $C$.

By the closed map lemma $\pi(\cl(C))=\cl(\pi(C))=J$, hence $C$ 
is of type 
(IIa).
\end{proof}

Consider a relation defined on $M$,
\[\label{relation} x\sim_{es} y \in M \quad\text{if}\quad x,y 
\in M_c^{es}, \quad\text{for some}\quad  c\in \text{V}_{cr}(f).\]
The relation induces an equivalence relation on $M$. We will denote
the quotient space by $M_{es}$. Let $\pi_{es}\colon  M\to M_{es}$ be the 
canonical projection. Define $f_{es}\colon  M_{es}\to \mathbb{R}$ by
$f_{es}\big([x]\big)=f(x)$. We have $f_{es}\circ \pi_{es}=f$ and the
following lemma is a direct consequence of the definition.

\begin{lemma}\label{Reeb of f_{es}}
The Reeb graph of $f_{es}$, denoted by $\mathcal{R}_{es}(f)$
coincides with the Reeb graph of $f$.
\end{lemma}

Let $S(D)=D\times[-1,1]/_{\!\sim} $ denote the suspension of $D$. Define a map
\[\varphi \colon C\to S(D)\,\,\text{ by }\,\, p\longmapsto 
[h^{-1}(p)],\]
where $h$ is a diffeomorphism from Proposition 
\ref{prop:gradient_product_diffeo}.
 We can continuously extend the map to a map $\overline{\varphi}\colon 
\pi_{es} (\cl(C))\to S(D)$ by
\[\overline{\varphi}(p)=
\begin{cases}
\varphi(p),& \text{if } p\in C,\\
[x,1],& \text{if } p\in M_{c'},\\
[x,-1],& \text{if } p\in M_{c}.
\end{cases}
\]

Since $\overline{\varphi}$ is continuous, bijective and its domain
and image are compact and Hausdorff, $\overline{\varphi}$ is a homeomorphism.
The proof of the following proposition is a straightforward consequence of 
Proposition \ref{prop:gradient_product_diffeo}.

\begin{proposition}\label{suspension} Let $C$ be a component $C$ of type (IIa).
The function \[\overline{\varphi}\colon \pi_{es} (\cl(C))\to S(D)\]
is a homeomorphism.
\end{proposition}

\begin{remark}\label{simply-connected}
Since $D$ is path-connected, $S(D)$ is simply-connected. Given a point $p\in 
D$ which is not in $\alpha$\nobreakdash- or $\omega$\nobreakdash-set of a 
critical point, an embedded interval $\{p\}\times [a,b]$ is mapped to 
$\{p\}\times [-1,1]$ in $S(D)$, see Proposition \ref{thm:gradient flow}. 
Contracting the interval yields $\big(\Sigma (D),[(y,0)]\big)$, the reduced 
suspension 
of $D$ which is a well pointed simply-connected space.
\end{remark}

\begin{remark} It is worth of pointing out that our Proposition 
\ref{suspension} 
provides a proof of the following theorem of Reeb.

\begin{quotation}
\textbf{Theorem.} \textit{Let $M$ be a smooth manifold. Suppose that there 
exists a smooth function $f\colon  M\to \mathbb{R}$ with exactly two 
non-degenerate critical points. Then $M$ is homeomorphic to the sphere $S^m$.}
\end{quotation}

Indeed, in the situation the Reeb graph $\mathcal{R}(f) $ consists of just one 
edge 
connecting two vertices corresponding to the maximum and minimum (see also 
Lemma \ref{local maximum}). Consequently $M=\pi_{es} (\cl(C))$ is homeomorphic 
to the suspension of a level set $D$. Since critical points of $f$ are 
non-degenrated, it follows that $D$ is homeomorphic to a sphere $S^{m-1}$. The 
result follows. Note that without the assumption on non-degeneratedness, we 
cannot assume that $D$ is homeomorphic to a sphere. For example, degenerated 
maximum may have neighbourhoods homeomorphic to $C\Sigma$, the Brieskorn 
variaty (i.e. the cone on Brieskorn spheres).

In general, since $D$ and $S(D)$ are manifolds, it follows from the Poincare 
duality that $D$ is a $(m-1)$\nobreakdash-homology sphere. Thus the suspension 
of $D$ is simply-connected homology sphere hence by Poincare conjecture it is 
homeomorphic to $S^{m}$.
\end{remark}

\subsection{Families of paths}
\begin{definition}\label{def:homotopy_rec}
Let $c,c'$ be two critical values of $f$ and
let $\gamma_1,\gamma_2$ be two paths, such that
\begin{align*}
f(\gamma_1(0))= f(\gamma_2(0))&=c' \quad \text{(that is }\gamma_1
(0), \gamma_2(0)\in
M_{c'}^{es}\text{), and }\\
f(\gamma_1(1))= f(\gamma_2(1))&=c \,\;\!\quad \text{(that is }\gamma_1
(1), \gamma_2(1)\in M_c^{es}\text{).}
\end{align*}

\begin{itemize}
\item  We say that a map $H\colon I\times I \to M$ is a homotopy 
\textbf{relative to essential components (\rec\!)} between $\gamma_1$ and 
$\gamma_2$ when:
\begin{enumerate}
\item $H$ is a homotopy between $\gamma_1$ and $\gamma_2$.
\item  $H(0,s)\in M_{c'}^{es} $ and  $H(1,s)\in M_c^{es}$, for all $s\in I$.
\end{enumerate}

\item A \textbf{vertex-path} is a path $\varepsilon\colon I\to M$ contained in 
an essential component of a critical level.

\item  We say that a path $\gamma$ is \textbf{contractible \rec}\! if there 
exists a homotopy \rec from $\gamma$ to a vertex-path.
\end{itemize}
\end{definition}

Since level sets are path-connected, contraction \rec is equivalent to the 
existence of a homotopy \rec from $\gamma$ to a constant path
$\sigma(t)=p\in \textup{Cr}(f)$, for some critical point $p\in M^{es} _c$. Note 
that we can choose $p$ arbitrarily from the connected component of $M_c^{es}$ 
to which $\gamma(0)$ belongs.

\begin{definition}\label{def:edge-path}
Let $C$ be a component of type (IIa) of $M^{(c,c')}$. An
\textbf{edge-path of $C$} is a path $\gamma$ such that
\begin{itemize}
 \item $\gamma(0)\in M_{c'}^{es}$ and $\gamma(1)\in M_c^{es}$,
 \item $\gamma((0,1))\subset C$.
\end{itemize}
We say that an edge-path is \textbf{decreasing} if
\[f(\gamma(t)) < f(\gamma(t')) \text{ for } t>t'.\]
\end{definition}

\begin{lemma}\label{lem:homotopy lifting from SD to clC}
Let $C$ be a component of type (IIa). Suppose that two edge-paths $\gamma_0$ 
and 
$\gamma_1$ are homotopic via homotopy $H$ fixing the endpoints in $\pi_{es} 
(\cl(C))\cong S(D)$. Then they are homotopic \rec in $\cl(C)$.
\end{lemma}

\begin{proof}
Lifting of the homotopy in $\pi_{es} (\cl(C))$ to $\cl(C)$ amounts to finding 
appropriate path in $\cl(C)\setminus C$.

For every step $s$ of $H(t,s)$, the homotopy from $\gamma_0(t)$ 
to $\gamma_1(t)$, let $x_s \in M^{es} _{c'} $ denote the 
limit \[x_s=\lim_{t\to 0} H(t,s)\] taken in $\cl(C)$. This is a valid 
definition 
since $C$ is diffeomorphic to $S(D)$ with the top point and the bottom point 
removed.
By continuity of $H$, $\phi_{c'} (s)=x_s$ is a path in $M^{es} 
_{c'}$. The homotopy \rec in $\cl(C)$ from $\gamma_0$ to $\gamma_1$ can be 
obtained by pre-stacking $H\restrictedTo{C}$ with $\phi_{c'}$ and 
post-stacking with 
(similarly defined) $\phi_c$.
\end{proof}

\begin{lemma}\label{lem:edge-path to decreasing}
Let $\gamma$ be an edge-path of $C$. Then $\gamma$ is homotopic \rec to a 
decreasing edge-path.
\end{lemma}

\begin{proof}
We will find a homotopy (fixing the endpoints) in $S(D)$ from $\pi_{es} 
(\gamma)$ to a linear parametrisation of an interval connecting the top vertex 
and the bottom vertex. Then by Lemma \ref{lem:homotopy lifting from SD to clC} 
the proof is finished.

Let $\gamma$ be an edge-path of $C$ and set $D=M_d\cap C$, a connected 
component of a level set for some $d\in (c,c')$. Since $C$ is of type 
(IIa), $C$ is diffeomorphic (by Proposition \ref{prop:gradient_product_diffeo}) 
to $D\times(-1,1)$. 

Choose a point $y\in D$ and a consider $\Pi\colon S(D)\to 
\big(\Sigma(D),[(y,0)]\big)$, 
the contraction of the interval $\{y\}\times[-1,1]\subset S(D)$.
Observe that $\big(\Sigma(D),[(y,0)]\big)$ is the reduced suspension of $D$ 
which is a well pointed space. Under these contractions an edge-path $\gamma$ 
becomes $\Pi(\pi_{es} (\gamma))$, a loop based at $[(y,0)]$. Since $D$ is 
connected, $\big(\Sigma(D),[(y,0)]\big)$ is simply-connected, hence 
$\Pi(\pi_{es} (\gamma))$is contractible to $[(y,0)]$. 

To lift the contracting homotopy to $S(D)$ note that \[\Pi\colon S(D)\to 
\big(\Sigma(D),[(y,0)]\big)\] is a fibration such that the fibre has homotopy 
type of 
point (is a point or an interval to be precise). Since we can always lift the 
constant map $H(1,t)\equiv [(y,0)]$, the homotopy lifting property provides a 
homotopy $H\colon I^{2}\to S(D)$ such that 
\[\begin{cases}
H(t,0)&=\pi_{es} (\gamma(t)), \\ H(t,1)&\in \{y\}\times[-1,1], 
\end{cases} \quad \text{and}\quad
\begin{cases}
H(0,s)&\in \{y\}\times[-1,1],\\
H(1,s)&\in \{y\}\times[-1,1].
\end{cases}
\]
Note that $\{y\}\times[-1,1]$ is a fibre over point $[(y,0)]$. We want to 
modify $H$ to a homotopy $\widetilde{H}$ from $\pi_{es} (\gamma)$ to 
a linear path 
\[\alpha(t)=(y,-1\cdot t + 1\cdot(1-t))\]
starting at $(y,1)$ and ending at $(y,-1)$.

Observe that $H(0,s)=(y,r_0(s))$ and $H(1,s)=(y,r_1(s))$ for some
functions $r_i \colon [0,1]\to [-1,1]$. Let $h^r_1\colon [0,1]\to
\{y\}\times [-1,1]$ be the path from $(y,1)$ to $(y,r)$ given by
\[h^r_1(t)=(y,(1-t)1+t r_0).\]
Then $h_1^{r_0(s)}$ is a path which begins at $(y,1)$ and ends at 
$(y,r_0(s))=H(0,s)$. For each $s\in I$ pre-stack $H$ with $h_1^{r_0(s)}$. 
We remark that the collection of $\{h_0^{r(s)}\}_s$ forms a linear homotopy 
from the interval $\{y\}\times [r_0(0)=1,1]$ to $\{y\}\times[r_0(1),1]$, which 
keeps the initial point $(y,1)$ fixed. The same construction can be applied to 
obtain $h_{-1}^{r_1(s)}$, a path from $(y,r_1(0)=-1)$ to $(y,r_1(1))$. 
Post-stack $H$ with the whole homotopy given by 
$\big\{\big(h_{-1}^{r_1(s)}\big)^{-1}\big\}_{s}$ (this is the reversed 
homotopy).

Observe that for a single $s$, the pre- and post-stacked 
$H\restrictedTo{\{s\}}$ becomes a 
path from $h_1^{r_0(s)}(0)=(y,1)$ to $h_{-1}^{r_1(s)}(1)=(y,-1)$. Since we are 
stacking homotopies, the whole family of pre- and post-composed paths is 
continuous in $s$, and therefore is a homotopy 
\[\widetilde{H}\colon I^2 \to S(D),\]
from $\pi_{es} (\gamma)$ to $\{y\}\times[-1,1]$ which keeps the endpoints 
fixed.
\end{proof}

\begin{corollary}\label{cor:every-epaths-are-htpic}
Every two edge-paths of $C$ are homotopic \rec.
\end{corollary}

\begin{proof}
By the above lemma, such edge-paths are homotopic to a decreasing 
edge-path given by an interval $[(y,t)]$ in $S(D)$. Since every such intervals 
are homotopic \rec by a `linear' homotopy (see proof of Lemma 
\ref{lem:hmtpy-to-pth-crit2crit}), the corollary follows by application of 
Lemma \ref{lem:homotopy lifting from SD to clC}.
\end{proof}

\begin{lemma} \label{lem:hmtpy-to-pth-crit2crit}
Every decreasing edge-path of $C$ is homotopic \rec to
a decreasing edge-path of $C$ connecting two critical points.
\end{lemma}
In other words, every decreasing edge-path of $C$ is homotopic \rec 
to a decreasing edge-path $\gamma$ of $C$ such that $\gamma(-1)=x_{c'}  \in 
M_{c'} $ and $\gamma(1)=x_{c}\in M_{c} $ are critical points.

\begin{proof}
By virtue of Lemma \ref{lem:homotopy lifting from SD to clC} it is enough to 
prove the existence of such homotopy in $S(D)$.

Let $\gamma$ be a decreasing edge-path. From the argument in
proof of Lemma \ref{lem:edge-path to decreasing} it follows that 
$\pi_{es}(\gamma)$ is homotopic \rec to a longitudinal path $\gamma_0: [-1,1] 
\to S(D)$ defined as $\gamma_0(t)=[(y, t)]$. Since $y \in D$ is arbitrary, we
can assume that 
\[x_{c'} =\lim_{t\to -1} (y, t) \in M_{c'} \cap \cl(C)\]
is a critical point by Proposition \ref{prop:joining critical levels} (note 
that the limit is taken in $M$, and not in $S(D)$). By the same argument there 
exists a point $x \in D$ 
such that \[x_c =\lim_{t\to 1} (x,t)\in M_c\cap \cl(C).\] 

In the following part of the proof we construct a 'linear' homotopy from 
$\gamma_0$ to (defined below) $\gamma_1$. Let $\omega$ be a path in $D$ such 
that 
$\omega(0)= y$
and $\omega(1)= x$. 
Furthermore, let $\phi(t)$ be a non-decreasing function 
$[-1,1]\to [0,1]$, such that $\phi(t)\equiv 0$ on a small neighbourhood of 
$t=-1$ and $\phi\equiv 1$ near $t=1$. 
We define a path $\widetilde{\gamma_1}: 
[-1,1]\to S(D)$ by the formula
\[\gamma_1 = \big[\big(\omega(\phi(t)), t 
\big)\big].
\]
Note that $\gamma_1$ can be lifted to a (decreasing, by 
definition) path $\widetilde{\gamma_1}(t)\in \cl(C)$ such that
\begin{align*}
\widetilde{\gamma_1}(-1)& =\lim_{t\to -1}\big(\omega(\phi(t)), t \big)
=\lim_{t\to -1}(x,t)=x_{c'}, \\
\widetilde{\gamma_1}(1)& =\,\,\lim_{t\to 1}\,\big(\omega(\phi(t)), t \big)
= \,\,\lim_{t\to 1}\,(y,t)=x_{c}. 
\end{align*}
Finally, the formula \[\Lambda(t,s)= \big[\big(\omega(st),t\big)\big]\]
gives us a homotopy \rec between $\gamma_0$ and $\gamma_1$. By Lemma 
\ref{lem:homotopy lifting from SD to clC}, the initial path $\gamma$ is 
homotopic \rec to $\widetilde{\gamma_1}$ what ends the proof. 
\end{proof}

\begin{definition}\label{def:extended edge-path} Let $C$ be a component of type 
(IIa) of $M^{(c,c')}$. An {\textbf{extended edge-path}} of $C$ is a path
$\gamma\colon I\to M$ such that:
\begin{enumerate}
\item $\gamma(0)\in M_{c'}^{es}$ and $\gamma(1)\in M_{c}^{es}$,

\item There exist $t_0<t_1\in I$ such that
\begin{itemize}
\item if $f(\gamma(t))=c'$ then $t\leqslant t_0 $, and
\item if $f(\gamma(t))=c$ then $t\geqslant t_1 $, and
\item $\gamma(t)\cap M_d^{es} = \varnothing$, for $d\neq c,c'$.
\end{itemize}
\end{enumerate}
\end{definition}

\begin{lemma}\label{lem:extnd-pth-decomposition-vtx-edg-vtx}
Every extended edge-path $\gamma$ is homotopic \rec 
to an edge-path of the form
\[\gamma = \sigma*\alpha*\sigma',\]
where $\sigma'$, $\sigma$ are vertex-paths contained in
$M_{c'}^{es}$, $M_c^{es}$ respectively, and $\alpha$ is a strictly
decreasing edge-path.
\end{lemma}

\begin{proof}
Let $t'$ be the largest $t_0$ such that the condition of the above
definition holds. Then take $\widetilde{\sigma}'$ to be
(re-parameterized) $\gamma\restrictedTo{[0,t']}$. Similarly one may define
$t''$ as the smallest $t_1$ and set $\widetilde{\sigma}=
\gamma\restrictedTo{[t'',1]}$. Finally, take
$\widetilde{\alpha}=\gamma\restrictedTo{[t',t'']}$.

Observe that $\gamma$ is contained in $\cl(C)$. Moreover by condition (2) of 
the definition above
$\widetilde{\sigma}'$ intersects only $M_{c'}$ and
$\widetilde{\sigma}$ intersects only $M_c$.
The image $\pi_{es}(\widetilde{\sigma}')$ is a loop contained in a
finite wedge of cones, hence is contractible to the middle point
of this wedge. Therefore $\widetilde{\sigma}'$ and
$\widetilde{\sigma}$ are homotopic \rec to $\sigma'$ and $\sigma$,
vertex-paths contained in $M_{c'} $ and $M_c$ respectively.

Since all edge-paths are homotopic \rec to decreasing edge-paths,
the lemma is proved.
\end{proof}

\begin{corollary}\label{cor:extnd-epaths-hmtpic-to-decreasing}
Every extended edge-path of $C$ is homotopic \rec to a decreasing edge-path.
\end{corollary}
\begin{proof}
By Lemma \ref{lem:extnd-pth-decomposition-vtx-edg-vtx} an extended 
edge-path $\gamma$ could be written as $\sigma*\alpha *\sigma'$, where 
$\sigma,\sigma'$ are vertex-paths and $\alpha$ is
an edge-path. Since vertex-paths are homotopic \rec to constant
paths, and every edge-path is homotopic \rec to a decreasing
one, lemma follows.
\end{proof}

\begin{theorem}\label{thm:extnd-epaths-homotopic}
Classes of homotopy \rec of extended edge-paths are in bijection with edges 
of $\mathcal{R}(f)$. Moreover, each class of homotopy \rec contains a decreasing 
path joining two critical points.
\end{theorem}
\begin{proof}

Given an extended edge-path $\gamma$ we may homotope
\rec it to appropriate edge-path by using Corollary
\ref{cor:extnd-epaths-hmtpic-to-decreasing}. Since all
edge-paths of $C$ are homotopic \rec by Corollary 
\ref{cor:every-epaths-are-htpic} there is just one
class of homotopy \rec of extended edge-paths of $C$.
By Lemma \ref{lem:hmtpy-to-pth-crit2crit} we choose a decreasing
representative which connects two critical points.
Since component $C$ is of type (IIa) (by definition of edge-path),
$C$ corresponds to a unique edge of $\mathcal{R}(f)$ by Proposition 
\ref{prop:compIIa_is_edge}.
\end{proof}

\section{Graphs}\label{sec:Graphs}

Let $V(f)$ be the set of all classes of homotopy \rec of
vertex-paths $[\varepsilon]$, and  $E(f)$ be the set all classes of homotopy 
\rec of extended edge-paths.

\begin{definition}\label{Reeb by paths}
Let $\mathcal{G}(f)$ denote the set $\{V(f), E(f)\}$ and $V(f)$ ($E(f)$) will be
called the set of vertices (set of edges, respectively) of $\mathcal{G}(f)$.
\end{definition}

We say that two vertices $[\varepsilon_1]$ and  $[\varepsilon_2]$
are adjacent by edge if there exists an extended edge-path $\gamma\colon 
I\to M$ such that $[\gamma(0)]=[\varepsilon_1]$ and
$[\gamma(1)]=[\varepsilon_2]$ ($\gamma(0)$ and $\gamma(1)$ are considered as 
constant maps). Note that there might be many different edges between two 
vertices.

\begin{proposition}\label{graph1}
The pair $\mathcal{G}(f)=(V(f),E(f))$ defines a directed finite graph.
\end{proposition}
\begin{proof}
Given an extended edge-path $\gamma$ we consider the constant vertex-paths
$\varepsilon_1 \equiv \gamma(0)$ and $\varepsilon_2\equiv
\gamma(1)$. If $\gamma^{\prime}\in[\gamma]$, then paths
$\varepsilon^{\prime}_1 \equiv \gamma^{\prime}(0)$ and
$\varepsilon_2^{\prime}\equiv \gamma^{\prime}(1)$ are homotopic
\rec to $\varepsilon_1$, $\varepsilon_2$, respectively. 
We associate the class $[\gamma]$ with the pair of vertices (a 
$1$\nobreakdash-simplex) $([\gamma(0)],[\gamma(1)])$. This endows 
$\mathcal{G}(f)$ with a simplicial structure. 

By Corollary \ref{cor:finite-numbr-contcd-compnts} we have a finite 
number of components of type (IIa), hence by Proposition 
\ref{prop:compIIa_is_edge} we have a finite number of edges. In every class of 
homotopy \rec we may find a decreasing path $\gamma$ by Corollary 
\ref{cor:extnd-epaths-hmtpic-to-decreasing} and we orient the edge in 
$\mathcal{G}(f)$ accordingly. 
\end{proof}

We write $[\gamma]=\big([\gamma(0)],[\gamma(1)], C\big)$ to indicate $C$, the 
component 
of type (IIa), to which $\gamma$ belongs.

\begin{theorem}\label{thm:G(f)-PLhomeo-R(f)}
The graph $\mathcal{G}(f)$ is homeomorphic by a simplicial 
homeomorphism to 
$\mathcal{R}(f) $, the Reeb graph of $f$.
\end{theorem}

\begin{proof}
Consider the map $\mathcal{G}(f)\to \mathcal{R}(f) $
defined by
\begin{align*}
V(f) \ni[\varepsilon]& \mapsto \pi(\varepsilon(0)),\\
E(f)\ni[\gamma]=\big([\gamma(0)],[\gamma(1)],C\big) & \mapsto \pi(\cl(C)).
\end{align*}
On vertices the map is bijective by the definition of the Reeb graph.
By Theorem \ref{thm:extnd-epaths-homotopic}, classes of extended edge-paths
correspond bijectively to connected components of type (IIa). By Proposition
\ref{prop:compIIa_is_edge}, these components correspond bijectively to edges of 
the Reeb graph.
\end{proof}

\subsection{A realization of the graph $\Gamma(f)$ as a subspace of $M$}

In this section we construct a finite one-dimensional complex $\Gamma(f)\subset 
M$ which is homotopy equivalent to $\mathcal{R}(f)$. 
For each critical value $c$ and for each component $C$ of type (IIa) choose a 
point $x_c^C$ such that $x_c^C \in \cl(C)\cap M_c^{es}$.

\begin{definition}\label{def:graph_Gamma}
We will denote by $\Gamma(f)$ the 1-dimensional CW-complex, defined as follows.
\begin{enumerate}[i)]
\item $0$\nobreakdash-cells are critical points of $f$;
\item $1$\nobreakdash-cells are edges of a `spanning tree' 
$\Gamma^{es}_c \subset M^{es}_c$ (provided by part (3) of Proposition 
\ref{prop:connected-componets}), for every essential component in every level 
set. Moreover,
\item as additional $1$\nobreakdash-cells, in every component $C$ of 
type (IIa) we choose a decreasing edge-path $\gamma_C\in E(f) $ joining the 
chosen critical points. The existence of such $\gamma_C$ is proved in Lemma 
\ref{lem:hmtpy-to-pth-crit2crit}.
\end{enumerate}
\end{definition}

\begin{proposition}\label{prop:Reeb of Gamma equal to Reeb}

Consider the restricted function 
$f\raisebox{-1pt}{\big{|}}_{\Gamma}\colon\Gamma(f)
\longrightarrow \mathbb{R}$. Then there is a simplicial homeomorphism of the 
Reeb 
graph $\mathcal{R}_{f|_{\Gamma}}$ of
$f\raisebox{-1pt}{\big{|}}_{\Gamma}$ and $\mathcal{R}(f) $ which is orientation 
preserving.
\end{proposition}

\begin{proof}
Every point in $\Gamma(f)$ belongs either to the image of some $\gamma_C$, or 
to an essential component of a level set. Hence every point in 
$\mathcal{R}_{f|_\Gamma}$ is either a single point from the 
image of $\gamma_C$, or an image of the whole $\Gamma^{es}_c $.
Define $\psi\colon \mathcal{R}_{f|_\Gamma} \to \mathcal{G}(f)$ by 
\[\psi(p)=
\begin{cases}
t[x^C_c]+(1-t)[x^C_{c'}]&\text{if $p=\pi(\gamma_C(t))$ for some $t\in (0,1),$}\\
 [x^C_c] &\text{if $p= \pi(\Gamma^{es} _c\cap C)$.}
\end{cases}
\]

Recall that $[x_c^C]$ denotes the class of homotopy \rec of a constant 
path. Since we chose exactly one representative $\gamma_C$ from every $C$, and 
$\gamma_C$ is decreasing, $\psi$ is linear on edges. Moreover $\psi$ is 
bijective on vertices, hence the composition 
\[\mathcal{R}_{f|_\Gamma}\xrightarrow{\psi} \mathcal{G}(f)\to \mathcal{R}(f)\] 
is the 
required simplicial homeomorphism (the 
second map is given by Theorem \ref{thm:G(f)-PLhomeo-R(f)}).
Finally orientation of $\mathcal{R}(f) $ is given by $f=f/_{\!\sim}$ and 
corresponds to 
the orientation of $\mathcal{R}_{f|_\Gamma}$ induced by $f|_{\Gamma}$, thus 
these two orientations agree.
\end{proof}

\begin{remark}
As a matter of fact we have 
\begin{align*}
f\raisebox{-1pt}{\big{|}}_{\Gamma}\!\raisebox{-2pt}{\Big{/}}_{\!\!\!\!\sim} &= 
f/_{\!\sim}
\intertext{which induces the function} 
\widetilde{f} \colon\mathcal{R}(f) &\cong\mathcal{R}_{f|_\Gamma}\to \mathbb{R}.
\end{align*}
\end{remark}

\begin{corollary}\label{cor:homotopy equivalence}
The graph $\Gamma(f)\subset M $ is homotopy equivalent to the Reeb
graph $\mathcal{R}(f) $.
\end{corollary}

\begin{proof}
Note that the pre-image in $\Gamma(f)\to \mathcal{R}_{f|_\Gamma}$ of every 
point is either a single point or a spanning tree $\Gamma^{es}_c$.
The result follows easily.
\end{proof}

\section{Applications}\label{sec:applications}

Let $\iota \colon \Gamma(f)\hookrightarrow M$ be the embedding given by 
Definition \ref{def:graph_Gamma}. The composition of maps 
\[\pi \circ \iota\colon \Gamma(f) \to \mathcal{R}(f)\]
induces an isomorphism of the values of any homotopy functor, in particular of 
the homotopy and homology groups. Recall that $\pi_1(\Gamma(f))\cong 
\pi_1(\mathcal{R}(f) )$ which is a free group  $\mathbb{F}_r$ on $r\geqslant 0$ 
generators.

Moreover, we can work in the pointed category, since a choice of a
vertex  $x_0$ of $\Gamma(f)$ as the distinguished point descends to a unique
vertex  $\pi(x_0)$ of $\mathcal{R}(f) $ which we set as the distinguished point 
of the latter. Consequently the first homology group 
$H_1(\mathcal{R}(f) ;\mathbb{Z} )= \pi_1(\mathcal{R}(f) )_{ab} = \mathbb{Z}^r$, 
and by Universal Coefficient Theorem $H^1(\mathcal{R}(f);\mathbb{Z}) = 
\mathbb{Z}^r$. By the same argument $H_1(\mathcal{R}(f) ;G )= G^r$ for any 
coefficient group $G$, and dually $H^1(\mathcal{R}(f) ;R) = R^r$ for any 
coefficient ring $R$.

\begin{proposition}\label{composition is identity}
The  map $\pi \circ \iota\colon  \Gamma(f) \to \mathcal{R}(f) $
induces isomorphisms 
\begin{itemize}
\item on fundamental groups groups \[(\pi \circ \iota)_\# = \pi_\# \circ
\iota_\# \colon  \mathbb{F}_r \to \mathbb{F}_r,\]
\item on the first homology groups
\[ (\pi\circ \iota)_*= \pi_* \circ  \iota_* \colon \mathbb{Z}^r \to 
\mathbb{Z}^r,\] 
\item and on the first cohomology groups 
\[(\pi\circ \iota)^*  = \iota^*\circ
\pi^*\colon R^r \to R^r\] for every coefficient ring $R$.
\end{itemize}
\end{proposition}

\begin{corollary}\label{cor:main application}
\begin{enumerate}[i)]
\item If $\pi_1(M)$ is a finite group, \textup{(}e.g. $M$ is simply
connected\textup{)} then the Reeb graph $\mathcal{R}(f)$ of every function is a 
tree.
\item If $\pi_1(M)$ is an abelian group, or more general, a
discrete amenable group, the Reeb
graph $\mathcal{R}(f)$ of every function contains at most one loop.
\end{enumerate}
\end{corollary}

\begin{proof} If $\pi_1(M)$ is finite then $(\pi \circ \iota)_\# $ factorizes 
throughout a finite group, thus its image is a finite group and must not be an 
isomorphism.

Analogously, if $\pi_1(M)$ is abelian then $(\pi \circ \iota)_\# $ factorizes 
throughout an abelian group. Therefore its image is abelian and hence $(\pi 
\circ \iota)_\# $ can not be an isomorphism if $r\geqslant 2$. Since 
$\iota_\#$ is a monomorphism, $\pi_1(M)$ contains $\mathbb{F}_r$ for $r\geqslant 
2$ as a subgroup. This is impossible if $\pi_1(M)$ is a discrete finitely 
generated amenable group (cf. \cite{Nowak}).
\end{proof}

It is known that discrete nilpotent and solvable groups are amenable.

\begin{corollary}\label{cor:reeb examples}
Let $f\colon  X\to \mathbb{R}$, be any $C^1$\nobreakdash-function with isolated 
 
critical points.
\begin{itemize}
 \item If $X= S^n$ for $n\geqslant 2$, or $X=\mathbb{C} P^n, \mathbb{R} P^n$ 
for any $n$, then $\mathcal{R}(f)$ is a tree.
\item If $X=\mathbb{T}^n$ for any $n$ then $\mathcal{R}(f)$ is either a tree or 
is homotopy equivalent to a circle.
\item If $X=M_g$ a closed orientable surface of genus $g$ then $\mathcal{R}(f)$ 
contains at most $2g$ loops.
\item If $X=M_g$ a closed non-orientable surface of genus $g$ then 
$\mathcal{R}(f)$ contains at most $g$ loops.
\end{itemize}
\end{corollary}

\begin{proof} 
The first two points are a trivial application of Corollary \ref{cor:main 
application}. For the third point observe that the composition 
\[\mathbb{Z}^r \xrightarrow{\iota_*}\mathbb{Z}^{2g } 
\xrightarrow{\pi_*}\mathbb{Z}^r\] 
is an isomorphism, thus we have $r\leqslant 2g$. The same holds for the 
non-orientable case with $2g$ replaced by $g$.
\end{proof}

In the case when $f\colon  M_g  \to \mathbb{R}$ is a Morse function the number 
of loops in $\mathcal{R}(f) $ is equal to $g$ (see \cite[Lemma A]{Cole-Mc}).

Now we strengthen our Proposition \ref{cor:reeb examples} showing
that for any $C^1$\nobreakdash-function $f$ on a closed orientable surface 
$M_g$ the number of loops in $\mathcal{R}(f) $ is lower or equal to the genus 
$g$ of $M$. We start with the following observation.

\begin{lemma}\label{connectness after removing}
Let $[p]\in \mathcal{R}(f) $ be a point which is in the interior of an edge, 
such that $\mathcal{R}(f) \setminus \{[p]\}$ is path-connected. Then 
$M\setminus \pi^{-1}([p])$ is path-connected.
\end{lemma}

\begin{proof}
Let $\Gamma(f)\subset M$, be the one-dimensional complex as in Definition 
\ref{def:graph_Gamma}. Recall that the map 
$\pi\raisebox{-1pt}{\big{|}}_{\Gamma}\colon\Gamma(f) \to \mathcal{R}(f) $ 
contracts 
`spanning trees' in essential components and is a homeomorphism elsewhere (see 
Proposition \ref{prop:Reeb of Gamma equal to Reeb}). Therefore the homotopy 
equivalence $\Gamma(f)\to \mathcal{R}(f)$ restricts to a homotopy equivalence
\[\Gamma(f)\setminus \pi^{-1}([p])\to\mathcal{R}(f) \setminus \{[p]\}.\]
Since $\mathcal{R}(f) \setminus \{[p]\}$ is path-connected, $\Gamma(f)\setminus 
\pi^{-1}([p])$ is path-connected.

Take $x,y\in M\setminus \pi^{-1}([p])$.
Let $x',y'\in \Gamma(f)\subset M$ be two points such that $\pi(x)=\pi (x')$ and 
$\pi(y)=\pi (y')$. By Proposition \ref{prop:connected-componets} points $x$ and 
$x'$ ($y$ and $y'$) can be joined by a path $\alpha$ ($\beta$, respectively) in 
the level set of $x$ (of $y$, respectively).
Take a path $\gamma$ in $\Gamma(f)\setminus \pi^{-1}([p])$ connecting
$x'$ and $y'$. Then $\alpha\ast\gamma\ast\beta^{-1}$ is a path in  
$M\setminus \pi^{-1}([p])$ joining $x$ and $y$.
\end{proof}

Repeating the argument above several times we obtain the following corollary.

\begin{corollary}\label{cor:connectneess after several removing}
Let $[p_1],\ldots,[p_r]\in \mathcal{R}(f) $ be points which are in the interior 
of 
edges. If $\mathcal{R}(f) \setminus \big\{[p_1],\ldots,[p_r]\big\}$ is 
connected, 
then $M\setminus \pi^{-1}\big(\{[p_1],\ldots ,p_r]\}\big) $ is 
connected.
\end{corollary}

\begin{theorem}\label{estimation of loops by genus}
If $M_g $ is a closed orientable surface of genus $g$, then the number of 
loops in $\mathcal{R}(f) $ is less than or equal to $g$.
\end{theorem}

\begin{proof}
Let  $L\subset \mathcal{R}(f)$ be a loop in the Reeb graph and $J\subset L$ an 
edge. 
If we remove from $\mathcal{R}(f)$ a point $[p]\in \int(J)$, then 
$\mathcal{R}(f)\setminus\{[p]\}$ is connected. Suppose now that $\mathcal{R}(f) 
$ has $r$ loops. 
Then 
\[\mathcal{R}(f)\setminus \big\{[p_1],\ldots,[p_r]\big\}\] 
is still connected, where 
each $p_i$ belongs to an edge in a unique loop. By Lemma \ref{cor:connectneess 
after several removing} \[M\setminus \pi^{-1}\big(\{[p_1],\ldots,[p_r]\}\big)\] 
is connected.

When $M_g$ is a compact surface, a connected component of the preimage of a 
regular point is a connected manifold of co-dimension $1$, i.e. is 
homeomorphic to $S^1$. Thus after removing $r$ (disjoint) circles 
$M$ is still a connected space. This is possible only if $M_g$ has genus at 
least $r$.
\end{proof}

The above theorem supplies a criteria to recognize a function which is not 
Morse. 

On the other hand, most manifolds admit smooth functions with very few 
critical points. It is well known that on every oriented closed surface $M_g $ 
for $g \geqslant 1$ there exists a $C^1$\nobreakdash-function $f\colon  M_g  
\to \mathbb{R}$ with only three critical points (see \cite[p. 90-91]{Seifert}). 
There 
are also manifolds of higher dimensions which admit $C^1$\nobreakdash-functions 
with only three critical points. An example of such function, is the real or 
complex projective space of (projective) dimension $3$ (see \cite{Milnor} for 
more information).

\begin{theorem}\label{Reeb for three critical}
Let $f\colon  M \to \mathbb{R}$ be a $C^1$\nobreakdash-function with only three 
critical points on a closed manifold. Then the Reeb graph $\mathcal{R}(f) $ is a 
tree 
with two edges.
\end{theorem}

We begin with the following lemma.

\begin{lemma}\label{local maximum}
Let $p\in M$ be a local extremum of a $C^1$\nobreakdash-function  
$f\colon M\to \mathbb{R}$ with finite number of critical points.
Then $[p]\in \mathcal{R}(f) $ is a vertex of degree $1$.
\end{lemma}

\begin{proof}
We assume that $p$ is a local maximum. The proof for a local minimum
is analogous. 

Suppose that $p\in M$ corresponds to a vertex of degree greater or equal than 
$2$ and $f(p)\in (d,c)$. This implies that every sufficiently small 
neighbourhood $U$ of $p$ satisfies
\[U\cap \left(M^{(d,c)}\setminus\{p\}\right)=X\sqcup Y,\]
for some non-empty and disjoint sets $X$ and $Y$.
Let $h\colon B\to \mathbb{R}^n$ be a chart around $p$ taking $p$ to $0$. If we 
take 
$U=B$ for a sufficiently small $B$ (i.e. $B\subset M^{(d,c)}$), the equality 
becomes \[B\setminus\{p\} = X\sqcup Y\] for some (possibly different, but still 
disjoint and non-empty) $X$ and $Y$. Then 
\[\mathbb{R}^n\setminus\{0\}=h(B\setminus\{p\})=h(X\sqcup Y)=h(X)\sqcup h(Y),\]
which is clearly a contradiction since $\mathbb{R}^n\setminus \{0\}$ is 
connected for 
$n\geqslant 2$.
\end{proof}

\begin{proof}[Proof of Theorem \ref{Reeb for three critical}] Let $p^+,p^-,q$ 
denote respectively the maximum, the minimum, and third critical point of 
$f\colon  M\to \mathbb{R}$. Let $f(p^+)= c'$, $f(p^-)= c$ , and $f(q)= d$
for some $c\leq d \leq c'$. By Lemma \ref{local maximum} there 
exists a unique connected component  $C^+$ ($C^-$) of type (IIa) such that $p^+
\subset \cl(C^+)$ ($p^-\subset \cl(C^-)$, respectively). Evidently $q\in 
\cl(C^+)$ and $q\in\cl(C^-)$. Otherwise $q$ would be in an isolated 
component of a connected manifold $M$. Consequently $\mathcal{R}(f) $ 
consists of three vertices $[p^+],[p^-], [q]$, and two edges
$J^+= [C^+]$, and $J^-=[C^-]$ which proves the statement.
\end{proof}

\subsection{Final remarks}
The results of this work show that the Reeb graph of a
function $f\colon M\to \mathbb{R}$  is a not very fine combinatorial
invariant of a the pair $(M,f)$. More information can be extracted from 
$\mathcal{R}(f) $ only if the fundamental group of $M$ contains a free group 
$\mathbb{F}_r$, 
on $r\geq 2$ generators. For example this works well for surfaces.  
This explains why the reverse representation of a surface constructed from the 
Reeb graph is attainable (see \cite{Masumoto,Sharko}). On the other hand in the 
case of surfaces one could expect to get more information by studying the 
embedding of $\Gamma(f)$ in $M$. In particular one can study the
Bollobas-Riordan-Tutte polynomial of $\Gamma(f)\subset M$ and ribbon graphs.

It is a natural question what is the image of $F_r=
\pi_1(\mathcal{R}(f))= \pi_1(\Gamma(f))$ in $\pi_1(M)$. 
A similar question has been answered in terms of horizontal homology by Dey and 
Wang in \cite{Dey}. Furthermore, one can reverse the question and ask whether 
every subgroup of $\pi_1(M)$ isomorphic to $\mathbb{F}_r$ is in the image of 
$\pi_1(\mathcal{R}(f))$ for some function $f$.

Usually we want to simplify the manifold using $\mathcal{R}(f)$, so there is a 
natural tendency to seek for functions with the least possible number of 
critical points. However, when trying to infer some information about $M$ from 
its Reeb graph, we should try to make $\mathcal{R}(f)$ as complicated 
as possible, hence there is a need for construction of complicated functions on 
$M$ (with plenty of critical points and loops). To our knowledge there is no 
algorithm that provides such functions. A recent paper by Benedetti and Lutz 
\cite{Benedetti} shows that even random Morse PL-functions tend to have a very 
simple Morse vectors.

{
\end{document}